\documentclass[12pt]{amsart}

\usepackage{amssymb}

\usepackage{enumitem}

\makeatletter
\@namedef{subjclassname@2020}{%
  \textup{2020} Mathematics Subject Classification}
\makeatother

\newtheorem{lem}{Lemma}[section]
\newtheorem{prop}{Proposition}

\newtheorem{defi}{Definition}[section]
\newtheorem{coro}{Corollary}[section]

\newtheorem{theo}{Theorem}

\newcommand{\dist}{{\bf dist}}

\newcommand{\C}{{\bf C}}
\newcommand{\D}{{\bf D}}

\newenvironment{block}[1]{\trivlist\item[\hskip \labelsep{{#1}.}]}{\endtrivlist}

\newtheorem{com}{Comment}
\newtheorem{exa}{Example}

{\end{enumerate}\end{sc}}

\begin{document}


\baselineskip=17pt


\title{On the Cauchy transform vanishing outside a compact}

\author{Genadi Levin}
\address{Institute of Mathematics\\ Hebrew University of Jerusalem\\
Givat Ram\\
Jerusalem, Israel}
\email{levin@math.huji.ac.il}

\date{}

\begin{abstract}
Motivated by a problem in holomorphic dynamics, we present a certain generalization of the celebrated F. and M. Riesz Theorem.
\end{abstract}



\maketitle

\section{Introduction}\label{intro}
Given a finite complex measure $\nu$ with a compact support on $\C$, let
$$\hat\nu(z)=\int\frac{d\nu(w)}{w-z}$$
be the Cauchy transform of $\nu$. For the following facts, see e.g. \cite{garnett}: $\hat\nu$ is locally in $L^1(dxdy)$, $\hat\nu$ exists almost everywhere on $\C$ and holomorphic outside
of the compact support supp($\nu$) of $\nu$, and $\nu(\infty)=0$. Moreover, if for an open set $U$ and an analytic on $U$ function $h$, $h=\nu$ $dxdy$-almost everywhere on $U$, then $|\nu|(U)=0$


The F. and M. Riesz Theorem asserts that, given a measure $\nu$ on the unit circle $S^1$,
if $\int_{S^1}w^n d\nu(w)=0$ for all $n=1,2,\cdots$, in other words, if the Cauchy transform $\hat\nu$ of $\nu$ vanishes outside the closed unit disk, then $\nu$ is absolutely continuous w.r.t. the Lebesgue measure on $S^1$.
Theorem~\ref{omeganu}, see below,
is a generalization of this theorem to the union of pairwise disjoint bounded finitely connected domains.

We use the following notations and terminology. Given a compact subset $K$ of the plane,
$A(K)$ is the algebra of all continuous function on $K$ which are analytic in the interior of $K$
and $R(K)$ is the algebra of uniform limits on $K$ of rational functions with poles outside $K$ ($=$uniform limits on $K$ of functions holomorphic on $K$). If $R(K)=A(K)$, we call $K$ a {\it A-compact}, and if a A-compact $K$ is nowhere dense, $K$ is a {\it C-compact} (C=continuous since in this case $A(K)=C(K)$, the set of all continuous functions on $K$).
Given an open bounded set $U$, let $A(U)$ be the algebra of all analytic in $U$ functions which extend continuously to $\overline U$.

Necessary and sufficient conditions for a compact in the plane to be A- or C-compact are given by Vitushkin \cite{Vit}. Here are some sufficient conditions \cite{gamelin}, \cite{Vit}.
$K$ is a C-compact if the area of $K$ is zero. $K$ is C-(respectively, A-)compact if every point of $K$ (respectively, every point of $\partial K$) belongs
to the boundary of a component of the complement $\C\setminus K$.
Besides, the boundary of a A-compact is always
a C-compact, see e.g. \cite{gamelin}, though not the opposite.

We would like to know about non-trivial measures $\nu$ supported on the boundary $\partial K$ of $K$ such that $\hat\nu=0$ outside $K$.
It is well known \cite{gamelin}, that $\hat\nu=0$ off $K$ if and only if $\int f d\nu=0$ for all $f\in R(K)$.
This implies that if $K$ is nowhere dense, such a non-trivial measure $\nu$ exists if and only if $K$ is not a C-compact.

Let us formulate a question motivated by a problem in holomorphic dynamics.
Given a nowhere dense compact $E$,
let $V=\cup_k\Omega_k$ be the union of some non-empty collection $\{\Omega_k\}$ of bounded components of the complement $\C\setminus E$. The problem we are interested in is the following:

Describe the sets $V$ and $E$ for which any measure $\nu$ supported on $E$ and such that
$\hat\nu=0$ in $\C\setminus (E\cup V)$ is in fact supported on $\cup_k\partial\Omega_k$ and absolutely continuous w.r.t. harmonic measures of $\Omega_k$.



We prove in Theorem \ref{omeganu}
that this is the case
if each component $\Omega_k$ is finitely connected without isolated points in its boundary and the following three conditions hold: (I) $\{\Omega_k\}$ is a so-called
D-collections, (II) $A(\overline V)=A(V)$, (III): (i) $\overline V$ is a A-compact, (ii) $E$ is a C-compact.

Let us comment about (I)-(III), see below for more details. In condition (I), $\{\Omega_k\}$
is said to be a D-collection if harmonic measures of different domains are mutually singular
and for each domain $\Omega$ of the collection, a holomorphic homeomorphism from a bounded circular domain $\Delta$ onto $\Omega$ extends
as a one-to-one map onto a subset of the full (arc) measure to $\partial\Delta$. Condition (II)
means that every continuous in $\overline V$ function which is holomorphic in $V$ is in fact holomorphic in the interior of $\overline V$.
This is obviously the case if Int$(\overline V)=V$. In condition (III), (ii) along with $\hat\nu=0$ off $E\cup V$ imply that
$\nu=0$ on $E\setminus\partial V$ (see Lemma \ref{local}), i.e.,
in fact $\nu$ is supported on $\partial V$. Note that (III)
holds,
for example, if the complement $\C\setminus E$ consists of a finitely many components.

If $E=S^1$, the unit circle, and $V=\D$, the unit disk, then (I)-(III) are satisfied and we recover the F. and M. Riesz theorem.

All conditions of Theorem \ref{omeganu}
turn out to be essential: the conclusion about the measure $\nu$ breaks down in general if one of the conditions (I), (II) or (III) does not hold, see Proposition \ref{condi-ii} along with Examples \ref{sv}-\ref{exIII}.


There exists the abstract F. and M. Riesz Theorem, see \cite{gamelin}, Theorem 7.6. It would be interesting to derive Theorem \ref{omeganu} from this result. The reason we proved Theorem \ref{omeganu}
was to apply it to a problem in holomorphic dynamics, see Corollary \ref{rings} and \cite{lll}.

After completing this note (arXiv:1911.05336, Nov 13, 2019) we found Erret Bishop's papers \cite{Bish1} and its sequel \cite{Bish2}. The main Theorem 3 of \cite{Bish2} is a particular case of
Theorem \ref{omeganu} of the present note in the case when the boundary $E$ of $E\cup V$ is equal to the boundary of the unbounded component of the complement to $E\cup V$ (it is fairly easy to see that this implies all conditions (I)-(III) to be hold). 
In the concluding Remark III of \cite{Bish2} E. Bishop asks whether the analog
of his Theorem 3 holds in a more general setting noting that "this seems to be a difficult question" and that it is clear that some extra hypotheses are necessary. Our Theorem \ref{omeganu} thus answers partially this question.

For another line of development of F. and M. Riesz Theorem and \cite{Bish1}-\cite{Bish2}, see more recent \cite{kha}, \cite{kha1} by Dmitry Khavinson.
Theorem \ref{omeganu} of the present paper is close in spirit to \cite[Theorem 1]{kha1}. While we study measures on $\partial X$ of a compact $X$ by uniformazing components of the interior of $X$, in \cite{kha}, \cite{kha1} the author takes an approach in which
measures on $\partial X$ are approximated from outside $X$. Theorem 1 of \cite{kha1} essentially states that given a compact $X$ for which the Dirichlet problem is always solvable, every measure $\mu$ on $\partial X$ such that $\hat\mu=0$ off $X$ is a weak-$^*$ limit of $\{\mu_n\}$ with $||\mu_n||\le ||\mu||$ and $d\mu_n=f_n(z)dz|_{\partial X_n}$, for any decreasing to $X$ sequence $X_n$($\supset X)$ of compacts with analytic boundaries and some $f_n\in R(X_n)$.

Let's remark finally that we try to keep the proofs as self-contained and elementary as possible.
\section{Statements}\label{st}
\subsection{Main result and (counter-)examples}
All measures unless stated otherwise are assumed to be complex and finite. Given two measures $\lambda$ and $\mu$ where $\mu$ is positive we write
$\lambda\ll\mu$
if $\lambda$ is absolutely continuous w.r.t. $\mu$, and
$\lambda_1\perp\lambda_2$
for two mutually singular measures $\lambda_1$,$\lambda_2$.
Given a bounded plane domain $\Omega$ (i.e., a connected open set of the plane) let $\omega_\Omega$
denote the harmonic measure on $\partial\Omega$ of the domain $\Omega$ w.r.t. a fixed point in $\Omega$.

Recall that a domain of the plane is {\it circular} if its boundary consists of a finite number of disjoint circles. An example is the unit disk $\D$. Given a finitely connected bounded domain $\Omega$ whose boundary contains no isolated points,
it is a classical result that there exist a bounded circular domain $\Delta_\Omega$ and a conformal homeomorphism
$$\psi_\Omega: \Delta_\Omega\to \Omega.$$
By the Fatou Theorem on radial limits, for (Lebesgue) almost every point $w\in \partial \Delta_\Omega$, the radial boundary value $\psi_\Omega(w)$ is defined.
\begin{defi}\label{d-coll} (cf. \cite{glick}, \cite{chrisbishop})
Given a (finite or infinite) collection $\{\Omega_i\}$ of pairwise disjoint finitely connected domains without isolated points on their boundaries,
we call $\{\Omega_i\}$ a {\bf D-collection} (D=Davie, see \cite{Da})
if: each $\psi_{\Omega_i}$ extends radially as a one-to-one map on a set of full measure of the boundary of $\Delta_{\Omega_i}$, and $\omega_{\Omega_i}\perp\omega_{\Omega_j}$ for $i\neq j$.
\end{defi}
Main result is
\begin{theo}\label{omeganu}.
Let $V$ be a bounded open set such that each component of $V$ is finitely connected without isolated boundary points. Let $\{\Omega_i, \kappa_i\}_{i=1}^N$, $1\le N\le \infty$ be a set of couples where $\{\Omega_i\}_{i=1}^N$ is a collection of all components of $V$ and, for each $i$,  $\kappa_i$ is a holomorphic function on $\Omega_i$.
For each $i$, let us fix a uniformization $\psi_{\Omega_i}:\Delta_{\Omega_i}\to\Omega_i$.
Let $E\subset\C$ be a compact set without interior points such that $E\subset\C\setminus V$ and $E\supset\partial V$.
\begin{enumerate}
\item [{\bf P1.}]
Assume that
\begin{enumerate}
\item [(I)] $\{\Omega_i\}_{i=1}^N$ is a D-collection,
\item [(II)] $A(V)=A(\overline V)$,
\item [(III)] (i) $\overline V$ is a A-compact, (ii) $E$ is a C-compact.
\end{enumerate}
Then (a) implies (b) where:
\begin{enumerate}
\item [(a)] there exists a measure $\nu$ supported in $E$
such that:
\begin{equation}\label{omegahatnu}
\hat\nu(z)=\left\{\begin{array}{ll}
\kappa_i(z) & \mbox{ if } z\in\Omega_i, \mbox{ for each } i \\
0 & \mbox{ if } z\in\C\setminus(E\cup V)
\end{array}
\right.
\end{equation}
\item [(b)] for every $i$,
\begin{equation}\label{omegakappa}
||\kappa_i||:=\limsup_{\epsilon\to 0}\int_{\partial\Omega_{i,\epsilon}}|\kappa_i(z)||dz|<\infty
\end{equation}
where $\Omega_{i,\epsilon}=\psi_{\Omega_i}(\{w: \dist(w,\partial\Delta_{\Omega_i})>\epsilon\})$.
Moreover,
\begin{enumerate}\label{concl}
\item [(b1)]
\begin{equation}\label{concl2}
\sum_{i=1}^N ||\kappa_i||<\infty,
\end{equation}
\item [(b2)] the following representation holds:
\begin{equation}\label{nurepres}
\nu=\sum_{i}^N \nu_i
\end{equation}
where $\{\nu_i\}_{i=1}^N$ are pairwise mutually singular measures, $\nu_i$ is a measure on $\partial\Omega_i$ such that $\nu_i\ll\omega_{\Omega_i}$ and
$||\kappa_i||=||\nu_i||$, the total variation of $\nu_i$. In particular, $\nu$ has no atoms.
\end{enumerate}
\end{enumerate}
\item[{\bf P2}.] Vice versa, assume that (\ref{omegakappa})-(\ref{concl2}) hold. Then there exists a measure $\nu$ supported on $\cup_{k=1}^N\partial\Omega_k$ such that
\begin{equation}\label{omegahatnuback}
\hat\nu(z)=\left\{\begin{array}{ll}
\kappa_i(z) & \mbox{ if } z\in\Omega_i, \mbox{ for each } i \\
0 & \mbox{ if } z\in\C\setminus\overline{V}
\end{array}
\right.
\end{equation}
Moreover, if harmonic measures on different $\Omega_k$ are mutually singular, then $\nu$ admits a representation (\ref{nurepres}) as in (b2).
\end{enumerate}
\end{theo}
Notice the case $E=\partial V$. Then III(i) implies III(ii).

Applying Theorem \ref{omeganu} with $E=\partial V$ and
$\kappa_i=1$ for all $i$ we get an answer to the Problem 4.2, p.55, \cite{garnett} for the sets $V$ that satisfy conditions (I)-(III): Let $V=\cup_k\Omega_k$ be a bounded open set such that its components $\{\Omega_k\}$ form a D-collection and $A(V)=A(\overline V)=R(\overline V)$. Then there is a measure $\mu$ on $\partial V$ such that
$\hat\mu=1$ in $V$ and $\hat\mu=0$ off $\overline V$ if and only if for each $k$ the linear measure $\Lambda(\partial\Omega_k)$ of $\partial\Omega_k$ is finite and $\sum_k\Lambda(\partial\Omega_k)<\infty$.

Given a measure supported on a compact in the closed unit disk $\overline\D$, the unit circle can be used as a "screen"
to kill the Cauchy transform of this measure off $\overline\D$:
\begin{exa}\label{sv} (A. Volberg)
Let $V=\D\setminus K$ where $\D=\{|z|<1\}$ and $K=[0,1]$. Let $\nu_K$ be a measure support in $K$.
Assume that $\hat\nu_K(e^{it})\in L^1(0,2\pi)$. On the unit circle $S^1=\{|z|=1\}$ with the one-dimensional Lebesgue measure (i.e., the arc length measure $dt$, $0\le t\le 2\pi$), we define a new measure $\nu_c$ so that $d\nu_c(e^{it})=h_c dt$ with the density
$$h_c(e^{it})=\frac{1}{2\pi}e^{it}\hat\nu_K(e^{it}).$$
Let $\nu$ be a (finite) measure on $\partial V$ which is defined as follows: $\nu=\nu_K$ on $K$ and $\nu=\nu_c$ on $S^1$. Then
\begin{equation}\label{exa}
\hat\nu(z)=\left\{\begin{array}{ll}
\hat\nu_K(z) & \mbox{ if } z\in V, \\
0 & \mbox{ if } z\in\C\setminus\overline{V}
\end{array}
\right.
\end{equation}
Indeed, for $z\in\C\setminus (K\cup S^1)$,
$$\hat\nu_c(z)=\int_{S^1}\frac{d\nu_c(w)}{w-z}=
\int_K\frac{d\nu_K(u)}{u-z}\int_{S^1}\frac{dw}{2\pi i(w-z)(u-w)}$$
where the inner integral is equal to $-1$ for $|z|>1$ and $0$ for $z\in\D\setminus K$.
The same calculation, hence, (\ref{exa}) as well, hold if $K\subset\overline\D$ is any compact such that the length of $K\cap\partial\D$ is zero and $\nu_K$ is any measure on $K$
such that $\hat\nu_K(e^{it})\in L^1(0,2\pi)$.
\end{exa}
\begin{prop}\label{condi-ii}
Conditions (I)-(III) of Theorem \ref{omeganu} are necessary for its conclusion: there exist open sets $V$ and measures $\nu$ supported on $E:=\partial V$ such that $V$ is simply connected,
$\hat\nu=0$ off $\overline V$ and $\nu$ has atoms (so not absolutely continuous w.r.t. harmonic measure on $\partial V$) while, in notations of Theorem \ref{omeganu}, for the sets $V$ and $E$ one and only one condition (I), (II), (III) breaks down.
\end{prop}
\begin{proof}
Let $K$ be any nowhere dense compact as in Example \ref{sv} such that $V=\D\setminus K$ is simply connected. Let $E=K\cup S^1=\partial V$.
Since $\overline V=\overline{\D}$, $\overline V$ is a A-compact and $E=\partial V$ is a C-compact, so the condition III holds. On the other hand, taking $\nu_K$ a discrete measure with supp$(\nu)=K$ we see that the conclusion of Theorem~\ref{omeganu} does not hold. It means that at least one of the conditions I-II has to break down.
In particular, for $K=[0,1]$, condition I does not hold while II does so the condition I is necessary indeed.
As for the necessity of condition II, we choose $K=J$ where $J$ is a Jordan arc
such that $V=\D\setminus J$ is a simply connected domain which satisfies the condition I, i.e., the Riemann map
$\psi_V: \D\to V$ extends to a one-to-one map on a set of a full (arc) measure on $S^1$.
Hence, the condition II cannot hold in this case (this follows also directly from Theorem 1' of \cite{BW}, see also recent \cite{bcgj}, \cite{chrisbishop100}).
The existence of such Jordan arc $J$ follows from
Browder and Wermer \cite{BW}. Indeed, in \cite{BW} an example of a Jordan arc $J$ is constructed such that the Riemann map $h:\C\setminus\overline{\D}\to\C\setminus J$ from the complement to the closed unit disk onto the complement to $J$, $h(\infty)=\infty$, extends one-to-one on a set of full Lebesgue measure
on $S^1$. Then it is easy to see that
if we take $J\subset\overline{\D}$, $J\cap S^1=\{1\}$, then $V=\D\setminus J$ satisfies the condition I. [Proof: since $V\subset\C\setminus J$, $h^{-1}(V)$ is a well defined simply connected bounded domain
with pairwise analytic boundary; hence, if $\beta: \D\to h^{-1}(V)$ is a Riemann map then $\psi_V:=h\circ\beta:\D\to V$ extends to a one-to-one-map on a set of a full measure on $S^1$.]

That the condition III is necessary as well, see the following example.
\end{proof}
\begin{exa}\label{exIII}
Let $J$ be the Jordan arc as in \cite{BW}, see the proof of Proposition \ref{condi-ii}.
One can assume that $J\subset\D\cup\{1\}$ and the endpoints of $J$ are $0$ and $1$.
Let $\{D_j\}_{j=1}^\infty$ be a collection of open disks in $\D\setminus J$ with pairwise disjoint closures such that each $D_j$ touches $J$ at precisely one point which is neither $0$ nor $1$, the set of all such points is dense in $J$ and such that
$$\sum_{j=1}^\infty \frac{r_j}{d_j}<\infty$$
where $r_j$ is the radius of $D_j$ and $d_j$ is the distance between $0$ and $D_j$. (It's not difficult to realize that such choice of disks is possible.) Define sets $V$ and $E$ as:
$V=\D\setminus(J\cup\cup_{j=1}^\infty\overline D_j)$ and
$E=\partial V=S^1\cup J\cup\cup_{j=1}^\infty S_j$ where $S_j=\partial D_j$. Since there are non constant continuous on the Riemann sphere functions which are holomorphic in $\C\setminus J$, by \cite{dolzh}, \cite{gamelin} the compact
$\overline V=\overline\D\setminus\cup_{j=1}^\infty D_j$ is not a A-compact, i.e.,
$R(\overline V)\neq A(\overline V)$.

Now we define the measure $\nu$ on $E$ as follows.
Let $\nu=\delta_0+\sum_{j=0}^\infty\nu_{j}$
where $\delta_0$ is the Dirac measure at $0$, $d\nu_0(w)=\hat\delta_0(w)dw$, $w\in S^1$, is a measure on $S^1$ and, for each $j$,
$d\nu_j(w)=-\hat\delta_0(w)dw$, $w\in S_j$, is a measure on $S_j$.
Since for every $j\ge 0$,  $||\nu_j||\le 2\pi r_j/d_j$, then
$$||\nu||\le 1+2\pi(1+\sum_{j=1}\frac{r_j}{d_j})<\infty.$$
Similar to Example \ref{sv}, we get
$\hat\nu=0$ off $\tilde E$ and $\hat\nu=\hat\delta_0$ in $V$. On the other hand, $\nu$ is not absolutely continuous w.r.t. harmonic measure of $V$ because $\nu$ has the atom at $0$.
Notice that $V$ is simply connected and, for the sets $V$, $E$ as above, conditions (I), (II) of Theorem \ref{omeganu} holds but (III) does not: $\overline V$ is not A-compact. Note that at the same time,
$E=\partial V$ is a C-compact as every point of $E$ is at the boundary of the component $V$ of the complement to $E$.
\end{exa}
\subsection{Local removability of C-compacts}
We need the following
\begin{lem}\label{local}
\begin{enumerate}
\item [(a)] Any closed subset of a C-compact is C-compact.
\item [(b)] Let $K$ be a nowhere dense compact in $\C$ and $\mu$ a measure on $K$. Suppose that for a neighborhood $W$ of a point $x\in K$, $K\cap\overline W$ is a C-compact and $\hat\mu=0$ on $W\setminus K$. Then $\mu$ vanishes on
$K\cap W$, i.e., $|\mu|(W)=0$.
\end{enumerate}
\end{lem}
\begin{proof} (a) follows from the fact that any continuous function on a closed subset of a compact extends to a continuous function on the whole compact.
Let us prove (b). Let $U$ be the union of all components of $\C\setminus K$ that intersect $W$.
Then $W\setminus K\subset U$ and $\hat\mu=0$ on $U$. One can assume that $\infty\in U$ as otherwise,
for $M(z)=1/(x_0-z)$ with some $x_0\in U$, we replace $K$ by $\tilde K=M(K)$ and $\mu$ by a measure $\tilde\mu$ on $\tilde K$ such that $d\tilde\mu(w)=wd\mu(x_0-1/w)$. Thus $\infty\in U$.
Now, for $|\mu|(W)=0$ it is enough to prove that for each $x\in K\cap W$ there is a neighborhood $W_x\subset W$ such that
for all continuous with compact support in $W_x$ functions $g$, $\int gd\mu=0$.
So fix $x\in K\cap W$ and choose $W_x=B(x,r)$ where $r>0$ is so that $B(x,2r)\subset W$.
Let $g$ be a continuous function on $\C$ which is compactly supported in $B(x,r)$.
Let $\hat K=\C\setminus U$. It is enough to prove that $g\in R(\hat K)$.
Indeed, assume that there is a sequence of rational functions $R_n$ with poles outside $\hat K$ converging uniformly on $\hat K$ to $g$. Perturbing some of $R_n$ if necessary one can further assume that
all $R_n$ have simple poles. If say $R_n(z)=P(z)+\sum_{j=1}^m \alpha_j/(z-b_j)$ where $P$ is a polynomial and all $b_j\in U$, then
$\int R_n d\mu=\int P d\mu + \sum_{j=1}^m \alpha_j\hat\mu(b_j)=0$ because $\hat\mu=0$ in $U$
and $\int z^n d\mu(z)=0$ for all $n>0$ (this is because $\hat\mu=0$ in a neighborhood of $\infty$).
Then $\int g d\mu=\lim_n\int R_n d\mu=0$.
It remains to show that $g\in R(\hat K)$. We use Bishop's theorem (11.8 of \cite{zalcman}):
{\it Given a compact $X\subset\C$ and a continuous on $\C$ function $f$, assume
that for each $z\in X$ there is a closed neighborhood $B_z=\{|w-z|\le \delta_z\}$, $\delta_z>0$ such that $f|_{X\cap B_z}\in R(X\cap B_z)$. Then $f\in R(X)$.} Applying this to the compact $\hat K$ and the function $g$, if $z\in \hat K\setminus \overline B(x,r)$, then, for $\delta_z=|z-x|-r>0$, $g|_{\hat K\cap B_z}=0\in R(\hat K\cap B_z)$. On the other hand, for
$z\in\overline B(x,r)$ and $\delta_z=r$, $g|_{\hat K\cap B_z}\in R(\hat K\cap B_z)$ because
$\hat K\cap B_z\subset\hat K\cap\overline B(x,2r)\subset K\cap\overline W$ while $K\cap\overline W$
is a C-compact.
Notice that $g\in R(\hat K)$ follows also from Vitushkin's necessary and sufficient condition
for a function to be in $R(X)$, \cite{Vit}.
\end{proof}
\subsection{A particular case: rotation domains}
Let $\{\Omega_i\}$ be a collection of pairwise disjoint Jordan domains such that $\overline\Omega_i\cap\overline\Omega_j$ is at most a single point for all $i\neq j$, $V=\cup_i\Omega_i$ is bounded. Then obviously the conditions I-II of Theorem~\ref{omeganu} hold.

Here is a more interesting case which is originated in holomorphic dynamics. Recall that a simply or doubly connected domain $A$ is called a rotation domain for a rational function $f$ of degree at least $2$  if $f^Q:A\to A$ is a homeomorphism for some $Q\ge 1$
which is conjugate to an irrational rotation: for a conformal homeomorphism $\psi_A:\Delta_A\to A$ where $\Delta_A$ is either a round disk or a round annulus, the conjugate mam $R:=\psi_A^{-1}\circ f^Q\circ\psi_A: \Delta_A\to\Delta_A$ is
an irrational rotation of $\Delta_A$. ($A$ is called a Siegel disk or a Herman ring depending on whether $\Delta_A$ is a disk or an annulus.)
\begin{prop}\label{twoproperties}
Let $\{A\}$ be a collection of different rotation domains of a rational function $f$, i.e., each $A$ is either a Siegel disk or a Herman ring.
Then $\{A\}$ satisfies the conditions (I)-(II) of Theorem~\ref{omeganu}, i.e.,

(1) for every $A$, there is a subset $\tilde{X}$ of $\partial \Delta_A$ of the full Lebesgue measure (length) on which $\psi_A$ is one-to-one,

(2) for any two different rotation domains $A_1,A_2$, the harmonic measures $\omega_{A_1}$ of $A_1$ and $\omega_{A_1}$ of $A_2$ are mutually singular.

(3) let $V_f:=\cup_{A\in\{A\}}A$, then the interior of $\overline V_f$ is equal to $V_f$, therefore, $A(V_f)=A(\overline V_f)$.
\end{prop}
Part (1) is an easy corollary of the following claim which is proved in \cite{PomRud}, Theorem 2 for the Siegel disk and its
proof holds with obvious modifications for the Herman ring:
\begin{lem}\label{radlim2cases}
Let $A$ be a rotation domain of a rational function $f$ such that $f(A)=A$ and $\psi_A: \Delta_A\to A$ as above. There are only two cases:

(i) All radial limits $\psi_A(w)$ are different,

(ii) There is a point $a\in\partial A$  such that $f(a)=a$, moreover, if $\psi_A(w_1)=\psi_A(w_2)$ for some $w_1,w_2\in\partial \Delta_A$
then $\psi_A(w_1)=\psi_A(w_2)=a$.
\end{lem}
\begin{proof}[of Proposition \ref{twoproperties}]
(1) follows at once from Lemma~\ref{radlim2cases} and Riesz's uniqueness theorem for bounded analytic functions.
Let's prove (2). It's enough to show that the following is impossible: $\omega_{A_1}(Y)>0$ and $\omega_{A_2}(Y)>0$ for $Y=\overline{A_1}\cap\overline{A_2}$ where $Y$ is a subset of a component $L_1$ of $\partial A_1$ as well as a component $L_2$ of $\partial A_2$. So by a contradiction assume this is the case. First, since $f^Q(L_i)=L_i$ for some $Q>0$ and $i=1,2$, $f^Q(Y)\subset Y$. Secondly,
by the Fatou theorem on radial limits, there is a set $\tilde{Y}\subset \partial\Delta_{A_1}$ of positive length $|\tilde{Y}|>0$ such that for all $w\in\tilde{Y}$, the radial limit $\psi_{A_1}(w)$ exists and in $Y$. Moreover, since $f^Q(Y)\subset Y$, one can assume that $R(\tilde{Y})\subset\tilde{Y}$ where $R=\psi_{A_1}^{-1}\circ f^Q\circ\psi_{A_1}:\Delta_{A_1}\to \Delta_{A_1}$ is an irrational rotation of $\Delta_{A_1}$. Since $|\tilde{Y}|>0$, we get that $\tilde{Y}$ has a full measure. Then $Y$ is a closed subset of $L_1$ of the full harmonic measure $\omega_{A_1}$, hence, $Y=L_1$. Since $\omega_{A_2}(Y)>0$, then $Y=L_2$ as well, i.e.,
$L_1=L_2$. Taking now a small disk $B$ around some $x\in L_1=L_2$, we see that all iterates $f^{jQ}B$, $j\ge 0$, stay in $\overline{A_1\cup A_2}$, which is possible only if $x$ in the Fatou set, a contradiction.
Moreover, that the interior of $\overline{V_f}$ coincides with $V_f$ is also proved by a very similar argument.
\end{proof}
Theorem \ref{omeganu} and Proposition \ref{twoproperties} immediately imply
\begin{coro}\label{rings}
Suppose $\mathcal{H}$ is a non empty collection of bounded rotation domains
of a rational function $f$.
Let $V=\cup\{A: A\in\mathcal{H}\}$, $E\subset\C\setminus V$ a nowhere dense compact set such that
$\partial V\subset E$, and $\nu$ be a measure supported on $E$ such that $\hat\nu=0$ off $E\cup V$. If $E$ is a C-compact and $\overline{V}$ is a A-compact, then $\nu$ is, in fact, supported on $\partial V=\cup_{A\in\mathcal{H}}\partial A$ and, for each  $A$, $\nu|_{\partial A}\ll\omega_A$.
In particular, $\nu$ is non-atomic. Moreover, the function $\hat\nu\circ\psi_A'$ is in the $H^1$-Hardy space, i.e.,
$$\limsup_{\epsilon\to 0}\int_{\{z\in\Delta_A: \dist(z,\partial A)=\epsilon\}}|\hat\nu\circ\psi_A'(z)||dz|<\infty.$$
\end{coro}
Conjecturally, $\overline{V}=\cup\{\overline{A}: A\in\mathcal{H}\}$ is {\it always} a A-compact.
\section{Proof of Theorem~\ref{omeganu}}\label{mainprop}
\subsection{Preparatory statements}
The proof is heavily based on some general results, mainly of $\sim$1960's, see Theorems \ref{H}, \ref{T} and \ref{appr}.
The first one is a consequence of the Hahn-Banach and the Riesz representation theorems, see \cite{garnett}:
\begin{theo}\label{H}(V.P. Havin)
Let $F\subset \C$ be compact and let $g$ analytic on $\C\setminus F$ and $g(\infty)=0$. There is a measure $\nu$ on $F$ such that
$g(z)=\hat\nu(z)$ for all $z\notin F$ if and only if there is $C_g$ such that for all functions $h$ which are analytic in a neighborhood of $F$,
$$|T_g(h)|\le C_g||h||_F$$
where $||h||_F=\sup_{z\in F}|h(z)|$ and
$$T_g(h)=-\frac{1}{2\pi i}\int_{\partial U}g(z)h(z)dz$$
where $U$ is any small enough neighborhood of $F$ such that $\partial U$ consists of a finitely many analytic curves that surround $E$ in positive direction. When this is the case we may take $C_g=||\nu||$, the total variation of $\nu$.
\end{theo}
\begin{theo}\label{T}(G.G. Tumarkin, \cite{Tumarkin})
Let $g$ be analytic in $\C\setminus S^1$ and $g(\infty)=0$. Then $g=\hat\eta$ for some measure
$\eta$ (with supp$(\eta)\subset S^1$) if and only if
$$\sup_{0<r<1}\int_{S^1}|g(r\theta)-g(\frac{\theta}{r})||d\theta|<\infty.$$
\end{theo}
\begin{com}
Notice a particular case when $g=0$ off $\overline\D$.
\end{com}
For completeness, we prove here this statement for the direction we need. So let $g=\hat\eta$, for a measure $\eta$ on $S^1$. Given $\zeta=e^{i\theta}, w=e^{it}\in S^1$ and $0<r<1$,
$$\frac{1}{w-r\zeta}-\frac{1}{w-\frac{\zeta}{r}}=-\frac{\zeta(r^2-1)}{w\zeta(1-\frac{r\zeta}{w})(\frac{rw}{\zeta}-1)}=
e^{-it}P_r(\theta-t)$$
where $P_r(\theta-t)=(1-r^2)/(1+r^2-2r\cos(\theta-t))$ in the Poisson kernel. Therefore,
$$\hat\eta(r\zeta)-\hat\eta(\frac{\zeta}{r})=\int_{0}^{2\pi}P_r(\theta-t)e^{-it}d\eta(e^it)$$
and
$$\int_{S^1}|g(r\zeta)-g(\frac{\zeta}{r})||d\zeta|\le\int_{0}^{2\pi}d\theta\int_{0}^{2\pi}P_r(\theta-t)d|\eta|(e^{it})$$
$$\le \int_{0}^{2\pi}d|\eta|(e^{it})\int_{0}^{2\pi}P_r(\theta-t)d\theta=2\pi||\eta||$$
where $||\eta||$ is the total variation of the measure $\eta$.

Next two statements must be well known, too, though we are not aware of precise references and will give for completeness proofs here, cf. \cite{Bish1}-\cite{Bish2}.
For the first one, recall the following definition. Let $D$ be a bounded or unbounded circular domain, i.e.,
$\partial D=S_1\cup S_2\cup...\cup S_p$ where $S_1,\cdots,S_p$ are pairwise disjoint circles. The Hardy space $H^1(D)$ is a set of all holomorphic in $D$ functions $F$ with $F(\infty)=0$ if $\infty\in D$ such that
$$||F||_{H^1(D)}:=\limsup_{S\in\mathcal{S}}\int_S|F(w)||dw|<\infty$$
where $\mathcal{S}$ is a collection of all circles $S\subset D$ in a small neighborhood of $\partial D$ that are concentric to one of $S_j$, $j=1,\cdots,p$. It is well known e.g. \cite{Fisher} that any $F\in H^1(D)$ has a non-tangential limit $F(w)$ at almost every $w\in\partial D$ w.r.t. the Lebedgue (arc) measure on $\partial D$.
We need also the following representation for $F\in H^1(D)$ assuming $D$ is bounded and $S_1$ is the outer boundary of $D$:
\begin{equation}\label{repr}
F=F_1+F_2+\cdots+F_p
\end{equation}
where $F_j\in D_j$, $D_1$ is a bounded domain with the boundary $S_1$ and $D_j$, $j=2,\cdots,p$, is an unbounded domain with $\partial D_j=S_j$. This representation follows essentially from the Cauchy formula, see \cite{Fisher}.
Note that each $F_j$ is holomorphic in a domain that contains all other $S_k$, $k\neq j$.
\begin{lem}\label{omegakappa1}
Let $\Omega$ be a finitely connected bounded domain with no isolated points of the boundary and $\kappa:\Omega\to\C$ is holomorphic. Assume that
\begin{equation}\label{kappa1}
||\kappa||:=\limsup_{\epsilon\to 0}\int_{\partial\Omega_\epsilon}|\kappa(w)||dw|<\infty
\end{equation}
where $\Omega_{\epsilon}=\psi_{\Omega}(\{w: \dist(w,\partial\Delta_{\Omega})>\epsilon\})$.
Then there is a measure $\nu^\kappa$ supported on $\partial\Omega$ such that $\hat\nu^\kappa(z)=\kappa(z)$ for $z\in\Omega$ and $\hat\nu^\kappa(z)=0$ for $z\ni\overline{\Omega}$. Furthermore,
$\nu^\kappa\ll\omega_\Omega$
with $||\nu^\kappa||=||\kappa||$.
\end{lem}
\begin{proof} Choose $\epsilon_n\to 0$ and given $n$ define a measure $\nu_n$ on $\partial\Omega_{\epsilon_n}$
by $d\nu_n(z)=\frac{1}{2\pi i}\kappa(z)dz$. By (\ref{kappa1}), $\sup_n||\nu_n||<\infty$.
Let $\nu^\kappa$ be a weak$^*$ limit of the sequence $\{\nu_n\}$. By the Cauchy formula,
$$\hat\nu^\kappa(z)=\lim_{n\to\infty}\frac{1}{2\pi i}\int_{\partial\Omega_{\epsilon_n}}\frac{\kappa(\zeta)d\zeta}{\zeta-z}$$
is equal to $\kappa(z)$ for $z\in\Omega$ and $0$ off $\overline\Omega$.
Denote $\psi=\psi_\Omega$, $\Delta=\Delta_{\Omega}$, $\Gamma_n=\{w: \dist(w,\partial\Delta)=\epsilon_n\})$
and $\tilde\kappa=(\kappa\circ\psi) \psi'$.
Note that (\ref{kappa1}) is equivalent to : $\tilde\kappa\in H^1(\Delta)$.
Let $\psi(w)$ and $\tilde\kappa(w)$, $w\in\partial\Delta$, denote also corresponding limits (existing almost everywhere) of $\psi(u)$ and $\tilde\kappa(u)$ as $u\to w$ non-tangentially.
We have to prove that $\nu^\kappa\ll\omega_\Omega$
For this,
it is enough to show that for any continuous compactly supported function $h$ on $\C$,
\begin{equation}\label{nunukappa}
\int h d\nu^\kappa=\frac{1}{2\pi i}\int_{\partial\Delta}h(\psi(w))\tilde\kappa(w)dw.
\end{equation}
This would imply that $\nu^\kappa\ll\omega_\Omega$
and that $||\nu_\kappa||=||\kappa||$.
 One can check (\ref{nunukappa}) separately for each component of $\partial\Delta$. Let us do this for the outer component $S_1$ of $\partial\Delta=S_1\cup\cdots\cup S_p$ (for other components, the proof is the same with straightforward modifications).
One can assume $S_1=S^1$, the unit circle. Since $\int h d\nu^\kappa=\lim_{n\to\infty}\int h d\nu_n$,
we have to check that, for $r_n=1-\epsilon_n$,
$$\lim_{n\to\infty}\frac{1}{2\pi i}\int_{|w|=r_n}h(\psi(w))\tilde\kappa(w)dw=\frac{1}{2\pi i}\int_{|w|=1}h(\psi(w))\tilde\kappa(w)dw.$$
As $\tilde\kappa\in H^1(\Delta)$, let $\tilde\kappa=\sum_{j=1}^p\tilde\kappa_j$ the corresponding representation
as in (\ref{repr}). Then
$$\int_{|w|=r_n}h(\psi(w))\tilde\kappa(w)dw=\sum_{j=1}^p\int_{|w|=r_n}h(\psi(w))\tilde\kappa_j(w)dw$$
Let $j>1$. Since $\tilde\kappa_j$, $j\neq 1$, is a holomorphic function in a domain that contains $S^1$,
$h(\psi(w))\tilde\kappa_j(w)$ is bounded in $\{r<|w|<1\}$ for some $r<1$, hence, one can apply the Lebesgue Dominated Convergence Theorem:
$$\lim_{n\to\infty}\int_{|w|=r_n}h(\psi(w))\tilde\kappa_j(w)dw=\int_{|w|=1}h(\psi(w))\tilde\kappa_j(w)dw.$$
It remains to check that
$$\lim_{n\to\infty}\int_{|w|=r_n}h(\psi(w))\tilde\kappa_1(w)dw=\int_{|w|=1}h(\psi(w))\tilde\kappa_1(w)dw.$$
Here we have to use that $\tilde\kappa_1\in H^1(\D)$, the Hardy space in the unit disk.
We have:
$$|\int_{|w|=r_n}h(\psi(w))\tilde\kappa_1(w)dw-\int_{|w|=1}h(\psi(w))\tilde\kappa_1(w)|\le$$
$$\int_{0}^{2\pi}|r_n h(\psi(r_n e^{it}))-h(\psi(e^{it})||\tilde\kappa_1(r_n e^{it})|dt
+\int_{0}^{2\pi}|h(\psi(e^{it}))||\tilde\kappa_1(r_n e^{it})-\tilde\kappa_1(e^{it})|dt\to 0.$$
Here are some details. Let $I_n$ be the first integral and $J_n$ be the second one.
Then
$$J_n\le \sup_{t\in [0,2\pi]}|h(\psi(e^{it}))|\int_{0}^{2\pi}|\tilde\kappa_1(r_n e^{it})-\tilde\kappa_1(e^{it})|dt\to 0$$
because $\tilde\kappa_1\in H^1(\D)$ (\cite{Koo}, ch.II,B,2$^o$).
Let's prove that $\lim_n I_n=0$. Let $I:=\limsup_n I_n$.
Since $r\psi(r e^{it})\to \psi(e^{it})$ as $r\to 1$ a.e. in $t$ and $h$ is continuous,
for every $\sigma>0$ there is $E_\sigma\subset [0,2\pi]$ such that $l(E_\sigma)>2\pi-\sigma$
(where $l$ is the lebesgue measure on $(0,2\pi)$) and
$r_n h(\psi(r_n e^{it}))-h(\psi(e^{it}))\to 0$ uniformly in $t\in E_\sigma$. Since also $\tilde\kappa_1\in H^1(\D)$,
$$\lim_n\int_{E_\sigma}|r_n h(\psi(r_n e^{it}))-h(\psi(e^{it})||\tilde\kappa_1(r_n e^{it})|dt=0.$$
Hence,
$$I\le 2\sup|h|\limsup_n\int_{F_\sigma}|\tilde\kappa_1(r_n e^{it})|dt$$
where $F_\sigma=[0,2\pi]\setminus E_\sigma$ so that $l(F_\sigma)\le \sigma$.
On the other hand,
$$|\int_{F_\sigma}|\tilde\kappa_1(r_n e^{it})|dt-\int_{F_\sigma}|\tilde\kappa_1(e^{it})|dt|\le
\int_{0}^{2\pi}\chi_{F_\sigma}|\tilde\kappa_1(r_n e^{it})-\tilde\kappa_1(e^{it})|dt\le$$
$$\int_{0}^{2\pi}|\tilde\kappa_1(r_n e^{it})-\tilde\kappa_1(e^{it})|dt\to 0$$
as $r\to 1$. Thus
$$I\le 2\sup|h|\int_{F_\sigma}|\tilde\kappa_1(e^{it})|dt.$$
By the absolute continuity of the Lebesgue integral, for every $\epsilon>0$ there is $\sigma>0$ such that
for every $F_\sigma\subset [0,2\pi]$ with $l(F_\sigma)<\sigma$,
$$I\le 2\sup|h|\int_{F_\sigma}|\tilde\kappa_1(e^{it})|dt<2\sup|h|\epsilon.$$
As $\epsilon>0$ is arbitrary, $I=0$.
\end{proof}
In the following, for a bounded function $h: X\to \C$, let
$$||h||_X=\sup_{x\in X}|h(x)|.$$
\begin{lem}\label{Tcoro}
Let $g\not\equiv 0$ be analytic in a bounded circular domain $\Delta$. Assume that there is $C_g>0$ such that for every function $h$ which is holomorphic in a neighborhood of $\partial \Delta$ and all $\epsilon>0$
small enough,
\begin{equation}\label{Tcorob}
|\int_{\Gamma_\epsilon}g(w)h(w)dw|\le C_g||h||_{\partial\Delta}
\end{equation}
where $\Gamma_\epsilon=\{w\in \Delta: \dist(w,\partial \Delta)=\epsilon\}$
Then: (a) there is a measure $\eta$ which is absolutely continuous w.r.t. the Lebesgue (arc) measure on $\partial\Delta$
such that $\hat\eta$ is $g$ in $\Delta$ and $0$ off $\overline\Delta$, and (b)
\begin{equation}\label{bounddelta}
||g||:=limsup_{\epsilon\to 0}\int_{\Gamma_\epsilon}|g(w)||dw|<\infty.
\end{equation}
Moreover, $||g||=||\eta||$, the total variation of $\eta$, and there are a sequence $\{h_j\}$ of locally analytic
on $\partial\Delta$ functions and a sequence $\epsilon_j\to 0$ such that $||h_j||_{\partial\Delta}\to 1$ as $j\to\infty$ and
\begin{equation}\label{sup}
\lim_{j\to\infty}\int_{\Gamma_{\epsilon_j}}g(w)h_j(w)dw=||\eta||=||g||.
\end{equation}
Conversely, (\ref{bounddelta}) implies (obviously) (\ref{Tcorob}).
\end{lem}
\begin{proof} (\ref{Tcorob}) means that conditions of the Theorem~\ref{H} of Havin are satisfied for the compact $F= \partial\Delta$ and the function which is $g$ in $\Delta$ and $0$ outside of $\overline{\Delta}$.
Hence, there exists a measure $\eta$ supported on $\partial{\Delta}$ such that $\hat\eta=g$ in $\Delta$ and $\hat\eta=0$ in $\C\setminus\overline{\Delta}$.
Let us prove (b) first. We can assume that
$\partial\Delta$ is the union of $S_1=S^1$ and a finitely many disjoint circles $S_k$, $k=2,\cdots,p$
inside the unit circle $S^1$. Let $\eta_k=\eta|_{S_k}$, $1\le k\le p$. If $w\in\Delta$ is close to the component $S_1=S^1$ of $\partial\Delta$, so that $w=r\zeta$, $\zeta\in S^1$, we have:
$$g(r\zeta)=(\hat\eta_1(r\zeta)-\hat\eta_1(\zeta/r))+\delta(r,\zeta)$$
where $\delta(r,\zeta)=\sum_{k=2}^p(\hat\eta_k(r\zeta)-\hat\eta_k(\zeta/r))$ tends to $0$ uniformly
in $\zeta$ as $r\to 1$ because $\hat\eta_k$ is analytic off $S_k$. Hence, applying Theorem~\ref{T} to the measure $\eta_1$ we get (\ref{bounddelta})
for a component $\Gamma_\epsilon$ which is near $S_1$. The proof for other components of $\Gamma_\epsilon$ is very similar. This proves (\ref{bounddelta}). In turn, (\ref{bounddelta}) means that the condition of Lemma \ref{omegakappa1} is satisfied where $\Omega=\Delta$ (so that $\Delta_\Omega=\Delta$ and $\psi_\Omega=id$) and $\kappa=g$.
By Lemma \ref{omegakappa1}, there exists a measure $\nu^g$ supported on $\partial\Delta$ and absolutely continuous w.r.t. the arc measure on $\partial\Delta$ such that $\hat\nu^g=g$ in $\Delta$ and $\hat\nu^g=0$ outside of $\overline\Delta$.
But then the Cauchy transform of a measure $\eta-\nu^g$ vanishes outside of $\partial\Delta$. Since $\partial\Delta$
is a C-compact, we conclude that $\nu^g=\eta$. This proves part (a) along with $||\eta||=||g||_{H^1}$.
It remains to find a sequence $\{h_j\}$ as in (\ref{sup}).
Note that since $\Delta$ consists of a finitely many components (which are circles), it is enough to find $\{h_j\}$
for each component separately. So let $S_b(R)=\{w=b+Re^{it}, 0\le t\le 2\pi\}$ be such a component.
As $g\in H^1(\Delta)$ and $g\not\equiv 0$, the non-tangential limit $\tilde g$ of $g$ exists and non-zero for almost every $w\in S_b(R)$. Note that for $w\in S_b(R)$, $|dw|=\frac{\alpha}{w-b}dw$ where $\alpha=\frac{R}{i}$.
Since $\tilde g\neq 0$ almost everywhere, the function $H(w)=\frac{\alpha}{w-b}\frac{|\tilde g(w)|}{\tilde g(w)}\in L^{\infty}(S_b(R))$.
By Luzin's theorem, given $\delta>0$, there is a continuous on $S_b(R)$ function $H^\delta$ such that $\sup_w|H^\delta(w)|\le\sup_w|H(w)|=1$ and
$l(\{w: H(w)\neq H^\delta(w)\})<\delta$ where $l$ is the Lebesgue (arc) measure on $S_b(R)$. In turn, let $h^\delta$ be a locally holomorphic on $S_b(R)$ function
such that $\sup_{w\in S_b(R)}|H^\delta(w)-h^\delta(w)|<\delta$. Then the sequence of functions $h_j:=h^{1/j}$, defined for all $j$ big enough, and a sequence $\epsilon_j$ tending to zero fast enough, work.

Here are details. Assuming for simplicity $S_b(R)=S^1$, for each $j\ge 1$ choose $\epsilon_j>0$ such that $|h^{1/j}(r_j e^{it})-h^{1/j}(e^{it})|<1/j$ for all
$t\in [0,2\pi]$ and for $r_j=1-\epsilon_j$. One write:
$$|\int_{|w|=r_j}g(w)h^{1/j}(w)dw-\int_{|w|=1}|\tilde g(w)||dw||\le A_j+B_j+C_j$$
where
$$A_j:=|\int_{|w|=r_j}g(w)h^{1/j}(w)dw-\int_{|w|=1}\tilde g(w)h^{1/j}(w)dw|\to 0$$
because
$$A_j\le \int_{0}^{2\pi}|r_j g(r_j e^{it})-\tilde g(e^{it})||h^{1/j}(r_j e^{it})|dt+
\int_{0}^{2\pi}|\tilde g(e^{it})||h^{1/j}(r_j e^{it})-h^{1/j}(e^{it})|dt\to 0$$
as $\int_{0}^{2\pi}|g(r e^{it})-\tilde g(e^{it}))|dt\to 0$ with $r\to 1$ and by the choice of $\epsilon_j$,
$$B_j:=|\int_{|w|=1}(\tilde g(w)h^{1/j}(w)-\tilde g(w)H^{1/j}(w))dw|\to 0$$
as $\sup_{|w|=1}|H^{1/j}(w)-h^{1/j}(w)|\to 0$ for $j\to\infty$
and
$$C_j:=|\int_{|w|=1}\tilde g(w)H^{1/j}(w)dw-\int_{|w|=1)}|\tilde g(w)||dw|=$$
$$|\int_{|w|=1}(\tilde g(w)H^{1/j}(w)-\tilde g(w)H(w))dw|\to 0$$
as $l(\{w: H(w)\neq H^{1/j}(w)\})\to 0$ for $j\to \infty$.
\end{proof}
Another result we are going to use belongs to the approximation theory. Note that we don't use it in full generality, see comments right after the statement.
\begin{theo}\label{appr}(A. Davie \cite{Da}, \cite{Da1}, Zhijian Qiu \cite{Qiu})
Let $U$ be a bounded open subset of $\C$ such that each of its components is finitely connected and the complement to $U$ contains no isolated points. Then the following conditions are equivalent:
\begin{enumerate}
\item [(i)] the collection of components of $U$ is a D-collection,
\item [(ii)] $A(U)$ is strongly boundedly pointwise dense in $H^\infty(U)$: each bounded analytic function $h$ on $U$ is a pointwise limit
of a sequence $h_n\in A(U)$ with $||h_n||_{\overline{U}}\le ||h||_U$.
\end{enumerate}
\end{theo}
Comments on Theorem~\ref{appr} and the way we apply it:
(1) We need the implication (i)$\Rightarrow$(ii) only.
(2) By Davie \cite{Da1}, (ii) is equivalent to a seemingly weaker statement {\it (ii'): each bounded analytic function $h$ on $U$ is a pointwise limit of a bounded sequence $\{h_n\}\subset A(U)$}.
(3) Davie \cite{Da} (see preceding \cite{BW} though) proved that (i) and (ii') are equivalent when every component of $U$ is simply connected.
(4) The case when all but finitely many components of $U$ are simply connected can be reduced easily to (3) with help of a simple geometric construction (by covering every finitely connected component
by finitely many simply connected ones using only inner smooth cuts and then applying a local criterium of \cite{GG} that $A(U)$ is pointwise boundedly dense in $H^\infty(U)$; see Lemma 2.1 of \cite{Qiu} for details).
(5) We employ Theorem~\ref{appr} only to functions $h$ which are non zero on a finitely many components $U_1,\cdots, U_m$ of $U$. This case can be reduced to (4) as follows. Firstly, since $h=\sum_{i=1}^m h\chi_{U_i}$, it is enough to prove the claim for each $h\chi_{U_i}$ separately, i.e., when $h$ is non zero only on a single component, say, $U_1$. Now, if $U_{2},U_{3},\cdots$ are all other components of $U$,
let us modify $U$ to get a bigger open set $\tilde U$ roughly
by joining to each $U_j$ ($j>1$) some components of $\C\setminus U_j$ disjoint with $U_1$ to turn it into a simply connected domain;
see details in the proof of Theorem 2.1, \cite{Qiu}. Then $h\chi_{U_1}$ is bounded analytic in $\tilde U$ and
$A(\tilde U)\subset A(U)$, and we apply the case (4) to $\tilde U$.

We need the following consequence of Theorem~\ref{appr}:
\begin{coro}\label{apprcoro}
Let $V$ and $(\Omega_i)_{i=1}^N$ be as in Part 1 of Theorem~\ref{omeganu}, i.e., conditions (I)-(II) hold and
$\overline{V}$ is a A-compact. Fix a finite $m$, $1\le m\le N$. Given $\epsilon>0$ small enough, let $h$ be a bounded analytic function in $D_\epsilon:=\cup_{k=1}^m\Omega_k\setminus\overline{\Omega_{k,\epsilon}}$.
Then given a positive sequence $\sigma_n\to 0$, there is a sequence of rational functions $R_n$ with poles outside of
$E_\epsilon:=\overline V\setminus\cup_{k=1}^m\Omega_{k,\epsilon}$ such that $||R_n||_{E_\epsilon}\le ||h||_{D_\epsilon}+\sigma_n$,
$R_n(z)\to h(z)$ for every $z\in D_\epsilon$
and $R_n\to 0$ in $V\setminus\cup_{k=1}^m\overline\Omega_k$.
\end{coro}
\begin{proof}
Since $\overline V$ is a A-compact and the boundary of $E_\epsilon$ is the disjoint union of $\partial V$ and a finitely many analytic curves $\partial\Omega_{k,\epsilon}$, $1\le k\le m$, then $E_{\epsilon}$ is a A-compact as well. This follows e.g. from Vitushkin's theorem~\cite{Vit}.
Now, given $\epsilon>0$, $\{\sigma_n\}$ and $h$ as in the condition of the statement, let us extend $h$ from $D_\epsilon$
to a (bounded analytic) function $h^1$ in
$V_\epsilon:=\cup_{k=1}^N\Omega_k\setminus\cup_{k=1}^m\overline{\Omega_{k,\epsilon}}$ by defining $h^1=0$ on $\cup_{k=m+1}^N \Omega_k$. In view of condition (I), by Theorem~\ref{appr}, there is a sequence $h_n\in A(V_\epsilon)$
such that $h_n(z)\to h^1(z)$ as $n\to \infty$ for all
$z\in V_\epsilon$ and $||h_n||_{\overline{V_\epsilon}}\le ||h^1||_{V_\epsilon}=||h||_{D_\epsilon}$. By the condition (II), $h_n\in A(\overline V_\epsilon)=A(E_\epsilon)$.
Since $E_\epsilon$ is a A-compact, for each $n$ there is a rational function $R_n$ with poles outside of $E_\epsilon$ such that $||R_n-h_n||_{E_{\epsilon}}<\sigma_n$. The sequence
$R_n$ is as required. Indeed, for any $z\in E_{\epsilon}$, $|R_n(z)-h^1(z)|\le |R_n(z)-h_n(z)|+|h_n(z)-h^1(z)|<\sigma_n+|h_n(z)-h^1(z)|$, hence, $\{R_n\}$ tends to $h^1(z)=h(z)$ on $D_\epsilon$ 
and to $h^1(z)=0$ on $\cup_{k=m+1}^N\Omega_k$. At the same time, $||R_n||_{E_\epsilon}\le ||h_n||_{E_{\epsilon}}+\sigma_n\le
||h||_{D_\epsilon}+\sigma_n$.
\end{proof}
\subsection{Proof of Theorem~\ref{omeganu}}
Let us start with a more difficult Part 1: assume that there is a measure $\nu$ on $E$ such that (\ref{omegahatnu}) holds and then prove (b).
First of all, since $\hat\nu=0$ off $E\cup\overline V$ and $E$ is a C-compact by condition III(ii), Lemma \ref{local} immediately tells us that $\nu=0$ on $E\setminus\partial V$. In other words, one can assume
from the beginning that
$$E=\partial V.$$
Note at this point that since $\overline V$ is a A-compact by III(i), its boundary $E$ is a C-compact, see \cite{gamelin}, p.227.

Now, let $\tilde C=C_g$ be the constant guaranteed by Theorem~\ref{H}, for the compact $E=\partial V$ and the function $g=g_\kappa$ where
\begin{equation}\label{gkappa}
g_\kappa(z)=\left\{\begin{array}{ll}
\kappa_i(z) & \mbox{ if } z\in\Omega_i, \mbox{ for each } i \\
0 & \mbox{ if } z\in\C\setminus \overline V
\end{array}
\right.
\end{equation}
Let us fix a collection of uniformizations $\psi_{\Omega_k}:\Delta_{\Omega_k}\to\Omega_k$,
where the circular domains $\Delta_{\Omega_k}$ are pairwise disjoint with their closures.
Write $\psi_k=\psi_{\Omega_k}$, $\Delta_k=\Delta_{\Omega_k}$.
Let
$$\tilde\kappa_k(w)=\kappa_k(\psi_k(w))\psi_k'(w)$$
for $w\in \Delta_k$.
Let us fix an arbitrary finite $m$, $1\le m\le N$, and prove the following

{\bf Claim}. {\it For every function $\tilde h$ which is holomorphic in a small enough neighborhood of $\Gamma:=\cup_{k=1}^m\partial\Delta_k$ and every $\epsilon>0$ small enough,
\begin{equation}\label{pr2}
|\sum_{k=1}^m\int_{\Gamma_{k,\epsilon}}\tilde\kappa_k(w)\tilde h(w)dw|\le \tilde C||\tilde h||_{\Gamma}
\end{equation}
where $\Gamma_{k,\epsilon}=\{w\in\Delta_k: \dist(w,\partial \Delta_k)=\epsilon\}$.}

{\it Proof of the Claim.} Fix any $\epsilon>0$ such that $\tilde h$ is holomorphic in a neighborhood of the set
$\cup_{k=1}^m\{w\in\Delta_k: \dist(w,\partial \Delta_k)\le\epsilon\}$. Then, for each $\tilde\epsilon\in (0,\epsilon)$ and $k\in\{1,\cdots,m\}$,
\begin{equation}\label{100}
\int_{\Gamma_{\tilde\epsilon}}\tilde\kappa_k(w)\tilde h(w)dw=\int_{\Gamma_\epsilon}\tilde\kappa_k(w)\tilde h(w)dw.
\end{equation}
Let $h=\tilde h\circ\phi_k$ where $\phi_k=\psi_k^{-1}:\Omega_k\to \Delta_k$. Let $\epsilon_0\in (\tilde\epsilon, \epsilon)$.
Then $h$ is defined in
$D_{\epsilon_0}=\cup_{k=1}^m\Omega_k\setminus\overline\Omega_{k,\epsilon_0}$. Moreover, $h$ is holomorphic and bounded in $D_{\epsilon_0}$.
Let $E_{\epsilon_0}=\overline V\setminus\cup_{k=1}^m\Omega_{k,\epsilon_0}$.
Fix a positive sequence $\sigma_n\to 0$ and, by Corollary~\ref{apprcoro}, find a sequence of rational functions $R_n$ with poles outside of $E_{\epsilon_0}$ such that $R_n(z)\to h(z)$ for all $z\in D_{\epsilon_0}$,
$R_n(z)\to 0$ for all $z\in\cup_{k=m+1}^N\Omega_k$ and $||R_n||_{E_{\epsilon_0}}\le ||h||_{D_{\epsilon_0}}+\sigma_n$ for all $n$.
Hence, for $k=1,\cdots,m$ and fixed $\tilde\epsilon$,
\begin{equation}\label{101}
\int_{\Gamma_{\tilde\epsilon}}\tilde\kappa_k(w)\tilde h(w)dw=\int_{\partial\Omega_{k,\tilde\epsilon}}\kappa_k(z)h(z)dz=
\lim_{n\to \infty}\int_{\partial\Omega_{k,\tilde\epsilon}}\kappa_k(z)R_n(z)dz.
\end{equation}
Consider a finite collection of closed analytic curves $A_{\tilde\epsilon}:=\{\partial\Omega_{k,\tilde\epsilon}\}_{k=1}^m$.
Let us complete it by a finite collection $B_{\tilde\epsilon}$ of another pairwise disjoint closed analytic curves in $\C\setminus
\cup_{k=1}^m\overline\Omega_k$ so that all together $A_{\tilde\epsilon}\cup B_{\tilde\epsilon}=\partial U_{\tilde\epsilon}$ where $U_{\tilde\epsilon}$ is a (small) neighborhood of $E$.
Note that $R_n\to h$ on $A_{\tilde\epsilon}$, $R_n\to 0$ on $B_{\tilde\epsilon}\cap\cup_{k=m+1}^N\Omega_k$ while $\hat\nu=0$ on $B_{\tilde\epsilon}\cap(\C\setminus\overline V)$. Therefore, by the choice of $\tilde C$,
$$\lim_{n\to\infty}|\sum_{k=1}^m\int_{\partial\Omega_{k,\tilde\epsilon}}\kappa_k(z)R_n(z)dz|
=\lim_{n\to\infty}|\int_{\partial U_{\tilde\epsilon}}g_\kappa(z)R_n(z)dz|
\le\limsup_n\tilde C||R_n||_{E_{\epsilon_0}}\le$$
$$\lim_n\tilde C(||h||_{D_{\epsilon_0}}+\sigma_n)=\tilde C||h||_{D_{\epsilon_0}}=\tilde C||\tilde h||_{\cup_{k=1}^m\{w\in\Delta_k: \dist(w,\partial\Delta_k)\le\epsilon_0\}}.$$
Thus, by the latter inequality along with (\ref{100}) and (\ref{101}), for every $\epsilon_0\in (0,\epsilon)$,
$$|\sum_{k=1}^m\int_{\Gamma_\epsilon}\tilde\kappa_k(w)\tilde h(w)dw|\le \tilde C||\tilde h||_{\cup_{k=1}^m\{w\in\Delta_k: \dist(w,\partial\Delta_k)\le\epsilon_0\}}.$$
But $\lim_{\epsilon_0\to 0}||\tilde h||_{\{w\in\Delta_k: \dist(w,\partial\Delta_k)\le\epsilon_0\}}=||\tilde h||_{\partial{\Delta_k}}$ because $\tilde h$ is continuous up to $\partial\Delta_k$. This proves (\ref{pr2}) and the Claim.

We proceed as follows. Since the closures of $\Delta_k$, $1\le k\le m$ (where $m$ is finite) are pairwies disjoint, given $k\in\{1,\cdots,m\}$,
the Claim immediately implies that the condition of Lemma \ref{Tcoro} holds where $\Delta=\Delta_k$ and $g=\tilde\kappa_k(w)$. We conclude there exist a measure $\eta_k$ which is absolutely continuous w.r.t. the Lebesgue on $\partial\Delta_k$, a sequence of functions $\{h_{k,j}\}_{j=1}^\infty$ locally analytic near $\partial\Delta_k$ and a sequence $\epsilon_j\to 0$ such that $||h_{k,j}||_{\partial\Delta_k}\to 1$ as $j\to\infty$ and
\begin{equation}\label{supk}
\lim_{j\to\infty}\int_{\Gamma_{k,\epsilon_j}}\kappa_k(\psi_k(w))\psi_k'(w)h_{k,j}(w)dw=||\eta_k||=||\kappa_k||,
\end{equation}
where $||\kappa_k||=\limsup_{\epsilon\to 0}\int_{\partial\Omega_{k,\epsilon}}|\kappa_k(z)||dz|$.
Now, apply the Claim with $\tilde h_j$ to be $\tilde h_{k,j}$ near $\partial\Delta_k$, $1\le k\le m$.
By Lemma~\ref{Tcoro} and the Claim:
$$\sum_{k=1}^m||\kappa_k||=\sum_{k=1}^m||\eta_k||=\sum_{k=1}^m\limsup_{\epsilon\to 0}\int_{\Gamma_{k,\epsilon}}|\kappa_k(\psi_k(w))\psi_k'(w)||dw|\le \tilde C.$$
Since $\tilde C$ is {\it independent} on $m$, this proves (\ref{concl2}), that is,
$\sum_{k=1}^N ||\kappa_k||<\infty$.
This allows us to finish easily the proof that (a) implies (b) as follows.
For each finite $k$, $1\le k\le N$, by Lemma~\ref{omegakappa1} (with $\Omega=\Omega_k$ and $\kappa=\kappa_k$) there is a measure $\nu^{\kappa_k}$ such that $\hat\nu^{\kappa_k}=\kappa_k$ in $\Omega_k$ and $\hat\nu^{\kappa_k}=0$ off $\overline{\Omega_k}$, moreover, $\nu^{\kappa_k}\ll\omega_{\Omega_k}$
and $||\nu_{\kappa_k}||=||\kappa_k||$.
Define $\nu_w=\sum_{k=1}^N\nu^{\kappa_k}$. Then $||\nu_w||=\sum_{k=1}^N ||\nu^{\kappa_k}||<\infty$.
Since by condition (I) harmonic measures of different $\Omega_k$ are singular, $\{\nu^{\kappa_k}\}_{k=1}^N$ are pairwise singular as well.
Now,
$\hat\nu_w=\sum_{k=1}^N\hat\nu^{\kappa_k}$ is equal to $\kappa_k$ in $\Omega_k$ for each $k$ and $0$ off $\overline{V}$.
Let us compare measures
$\nu$ and $\nu_\omega$. For the difference measure $\tau=\nu-\nu_\omega$, we have: $\hat\tau=0$ off $E$
where $E$ is a C-compact.
Hence, $\tau=\nu-\nu_\omega=0$, and we are done with the implication (a) implies (b). In fact, above considerations prove Part 2, too, using again Lemma \ref{omegakappa1}.

\subsection*{Acknowledgements}
We would like to thank
Alexander Volberg for communicating Example \ref{sv} and for useful comments, and the referee for the reference \cite{kha1}. The author acknowledges the support of ISF Grant No. 1226/17.


\normalsize


\begin{thebibliography}{[HD82]}




\normalsize
\baselineskip=17pt

\bibitem{Bish1} Erret Bishop, \emph{The structure of certail measures.} Duke Math. J. 25(2) (1958), 283-289
\bibitem{Bish2} Erret Bishop, \emph{Boundary measures of analytic differentials.} Duke Math. J. 27(3) (1960), 331-340
\bibitem{chrisbishop} Christopher J. Bishop, \emph{A characterization of Poissonian domains.} Ark. Mat. 29(1991), no 1-2, 1-24
\bibitem{bcgj} C. J. Bishop, L. Carleson, J. B. Garnett and P. W. Jones, \emph{Harmonic measures supported on curves.} Pacific Journak of Mathematics, 138 (1989), no 2, 233-236
\bibitem{chrisbishop100} Christopher J. Bishop, \emph{Constructing continuous functions holomorphic off a curve.} Journal of Functional  Analysis 82 (1989), 113-137
\bibitem{BW} A. Browder, J. Wermer, \emph{Some algebras of functions on an arc.} J. Math. and Mech., 12 (1963), no 1, 119-130
\bibitem{Da} A. M. Davie, \emph{Dirichlet algebras of analytic functions.} Journal of Functional Analysis 6 (1970), 348-356
\bibitem{Da1} Alexander M. Davie, \emph{Bounded approximation and Dirichlet sets.} Journal of Functional Analysis 6 (1970), 460-467
\bibitem{dolzh} E. P. Dolzhenko, \emph{On approximationon closed regions and on null-sets.} Dokl. Akad. Nauk SSSR 143 (1962), 771-774
\bibitem{GG} T. Gamelin and J. Garnett, \emph{Constructive techniques in rational approximation.} Amer. J. Math., 89 (1967), 932-941
\bibitem{gamelin} T. Gamelin, \emph{Uniform Algebras}, Prentice Hall, 1969.
\bibitem{garnett} John Garnett, \emph{Analytic Capacity and Measure.} Lecture Notes in Math. 297, Springer-Verlag, 1972
\bibitem{glick} L. Glicksberg, \emph{A remark on analyticity of function algebras.} Pacific J. Math. 13(1963), 1181-1185
\bibitem{Fisher} S. Fisher, \emph{Function Theory on Planar Domains.} John Wiley \& Sons, 1983
\bibitem{Qiu} Zhijian Qiu, \emph{On pointwise bounded approximation.} Acta Mathematica Sinica, English Series, 25 (2009), no.7, 1217-1222
\bibitem{kha} Dmitry Khavinson, \emph{Annihilating measures of the algebra $R(X)$.} Journal of Functional  Analysis 58 (1984), 175-193
\bibitem{kha1} Dmitry Khavinson, \emph{F. and M. Riesz Theorem, analytic balayage, and problems in rational approximation.} Constr. Approx. 4(1988), no.4, 341-356
\bibitem{Koo} P. Koosis, \emph{Introduction to $H_p$ Spaces.} Cambridge University Press, Cambridge, 1980
\bibitem{lll} G. Levin, \emph{Fixed points of the Ruelle-Thurston operator and the Cauchy transform.} https://arxiv.org/abs/2002.03430
\bibitem{PomRud} Ch.Pommerenke and B.Rodin, \emph{Intristic rotations of simply connected regions,II.} Complex Variables and Applications, v.4(1985), 223-232
\bibitem{Tumarkin} G. G. Tumarkin, \emph{On integrals of Cauchy-Stieltjes type.} Uspehi Mat. Nauk (N.S.)
11(1956), no 4(70), 163-166
\bibitem{Vit} A. G. Vitushkin, \emph{Analytic capacity of sets in problems of approximation theory.} Russian Math. Surveys 22(1968),139-200
    \bibitem{zalcman} Zalcman, L., \emph{Analytic Capacity and Rational Approximation.} Lecture Notes in Mathematics,
        50, 1968
\end{thebibliography}
\end{document}